\documentclass[11pt]{amsart}
\usepackage{amssymb,amsmath}

 \usepackage{fancyhdr} 
 \usepackage[all]{xy}

 \usepackage{color}
\definecolor{red-}{rgb}{1.0,0.0,0.0}
\definecolor{grey}{rgb}{0.6, 0.6, 0.6}
\definecolor{brown}{rgb}{0.5,0.2,0.0}
\definecolor{brown-}{rgb}{0.0,0.1,1.0}
\definecolor{green-}{rgb}{0.0, 0.6, 0.0}
\definecolor{gold}{rgb}{0.8,0.7,0.0}
\definecolor{black}{rgb}{0.0,0.0,0.0}
\definecolor{DarkGreen}{rgb}{0.0,0.3,0.2}
\definecolor{LightGreen}{rgb}{0.8,1.0, 0.8}
\definecolor{yellow}{rgb}{0.9,0.9,0.0}

\theoremstyle{plain}% default
 \newtheorem{thm}{Theorem}%[section]
 \newtheorem{cor}[thm]{Corollary}
 
 \newtheorem{prp}[thm]{Proposition}

\def\thetitle{Rational conic fibrations}
\begin{document}

\title{Rational conic fibrations\\
of sectional genus two}

\author
[A. Lanteri, R. Mallavibarrena]{Antonio Lanteri,
Raquel Mallavibarrena}
\address{Dipartimento di Matematica ``F. Enriques'',
Universit\`a degli Studi di Milano, Via C. Saldini, 50,  I-20133 Milano, Italy}
\email{antonio.lanteri@unimi.it}
\address{Departamento de \'Algebra, Geometr\'ia y Topolog\'ia,  Facultad de Ciencias Matem\'aticas,
Plaza de las Ciencias, 3 - Universidad Complutense de Madrid,
28040 Madrid, Spain}
\email{rmallavi@mat.ucm.es}

\maketitle

\begin{abstract}

Polarized rational surfaces $(X,\mathcal L)$ of sectional genus two ruled in conics are studied.
When they are not minimal, they are described as the blow-up of $\mathbb F_1$ at some points
lying on distinct fibers. Ampleness and very ampleness of $\mathcal L$ are studied in terms of
their location. When $\mathcal L$ is very ample and there is a line contained in $X$
and transverse to the fibers, the conic fibrations $(X,\mathcal L)$ are classified and
a related property concerned with the inflectional locus is discussed.

\medskip


\smallskip


\end{abstract}

\section{Introduction}\label{Intro}

In the study of osculation for projective varieties there is a large literature concerning scrolls over a curve.
The investigation of quadric fibrations over a curve, from the same point of view, was started in \cite{LMP}. In particular, the special
case of dimension two, namely surfaces fibered in conics over a curve, is in itself interesting in some respects:
for instance, though the expected maximal dimension of the osculating spaces does not distinguish them from
a general surface, the special structure they have implies that the curve cut out by any osculating hyperplane
is reducible, containing a fiber.
In \cite{LM} we studied the osculatory behavior of smooth surfaces fibered in conics over a curve, and,
as an application, we described the inflectional loci of a special class of such surfaces with sectional genus $g=2$,
that we called Castelnuovo surfaces.
This study oriented our interest towards a better understanding of how these surfaces are framed in the more
general context of the rational conic fibrations polarized by an ample line bundle of sectional genus two.
From this perspective, first we have to note that while in the context of very ample line bundles rational conic fibrations with $g=2$ coincide with rational surfaces
with $g=2$, this is no longer true in the wider setting of ample line bundles.
Actually, in Fujita's classification of
polarized surfaces with $g=2$ \cite[Theorem 15.2, pp.\ 122--123]{F}, they simply correspond to cases $6_0)- 6_2)$, up to blow-ups (see case $0)$ there),
but other rational surfaces appear (cases 7), 8)) besides conic fibrations.

\medskip

So, let $(X,\mathcal L)$ be a polarized rational surface of sectional genus two ruled in conics. If $X$ is minimal, then $X \cong \mathbb F_e$, a Segre--Hirzebruch surface of invariant $e \geq 0$,
and the fact that $\mathcal L$ is an ample line bundle of genus two immediately leads to determine its expression in terms of the generators of $\text{Pic}(X)$ and to get the bound $e\leq 2$.
If $X$ is not minimal, by means of elementary transformations, we can describe $X$ via a birational morphism $\eta:X \to \mathbb F_1$ as the blow-up of $\mathbb F_1$ at points
$p_1, \dots, p_{\mu}$, lying on distinct fibers, and then
\begin{equation}\label{1}
\mathcal L = \eta^* \mathcal L_0 - \sum_{i=1}^{\mu} e_i,
\end{equation}
where $\mathcal L_0 \in \text{Pic}(\mathbb F_1)$ is an ample line bundle of genus two and $e_i$ is the exceptional curve corresponding to $p_i$ for $i=1, \dots ,\mu$. Moreover, $\mu \leq 11$.
Conversely, the line bundle $\mathcal L$ defined by \eqref{1} in this situation is ample provided that $p_1, \dots, p_{\mu}$ are general enough, according to Yokoyama \cite{Y}. So, as a first thing
we explore the precise conditions that the points $p_1, \dots , p_{\mu}$ have to satisfy in order to ensure the ampleness of $\mathcal L$, and we make them explicit. They are expressed
in the first part of Theorem \ref{restrictions}.
However, doing the same to ensure the very ampleness of $\mathcal L$, in which case, necessarily, $\mu \leq 7$, looks more delicate.
Nevertheless, we succeed to do it thanks to Proposition \ref{veryample} (2),
Proposition \ref{converse},
and the characterization of the nefness and bigness of the anticanonical bundle we obtain by translating Demazure's almost general position condition for points in $\mathbb P^2$ \cite{De}
in terms of $\mathbb F_1$. This is summarized
in the last part of Theorem \ref{restrictions}.

\medskip

In connection with this problem,
one could refer to the paper by E. Bese \cite{Be}, who studied, more generally, the spannedness
and the very ampleness of line bundles of type \eqref{1}, when $X$ is obtained by blowing-up any Segre--Hirzebruch surface $\mathbb F_e$ and $\mathcal L_0$ is any ample line bundle on it.
However, some results of Bese specialized to the situation under consideration do not coincide with ours, unfortunately (see Remark 5 for details). Actually, according to \cite[Theorem 4.2]{Be},
it seems that if $\mathcal L_0$ has bidegree $(2,4)$,
seven points $p_i$'s could lie on a general curve of $\mathbb F_1$ of bidegree (2,2) without affecting the very ampleness of $\mathcal L$. But this is not true
in view of Proposition \ref{converse} and Remark 4.
Moreover, the last part of Theorem \ref{restrictions}
says that this is essentially the key difference. In fact
further slight discrepancies simply derive from the fact that while
the conditions we obtain characterize the very ampleness of $\mathcal L$, \cite{Be} only provides sufficient conditions for that.

\medskip

In the final Section we provide the complete list of conic fibrations with $g=2$, polarized by a very ample line bundle,
containing a line transverse to the fibers (Theorem \ref{gamma}). If $\mu \leq 6$ any such line, as well as those constituting the irreducible components of
the reducible fibers, is contained in the inflectional locus of $X$
embedded in $\mathbb P^N$ by $|\mathcal L|$. This is not true however when $\mu=7$, and in this case we make explicit
the condition ensuring that such a line is not contained in the inflectional locus (Proposition \ref{final}).

\section{Background}\label{back}

All surfaces considered in this paper are smooth and defined over the complex field.
We use the standard notation and terminology from algebraic geometry.
In particular, we recall that a {\it{polarized surface}} is a pair $(X,\mathcal L)$ consisting of a surface $X$
and an ample line bundle $\mathcal L$ on $X$. The  {\it{degree}} and the {\it{sectional genus}} of
$(X,\mathcal L)$ are defined by $d(X,\mathcal L):=\mathcal L^2$ and $g(X,\mathcal L):=1 + \frac{1}{2}
(K_X+\mathcal L)\cdot \mathcal L$, respectively.
We say that
a line bundle is {\it{spanned}} to mean that the corresponding invertible sheaf is generated
by global sections.
For any integer $e \geq 0$, we denote by $\mathbb F_e$ the
Segre--Hirzebruch surface of invariant $e$,
i.e. $\mathbb F_e = \mathbb P(\mathcal E)$, where
$\mathcal E = \mathcal O_{\mathbb P^1}\oplus \mathcal O_{\mathbb P^1}(-e)$.
By $s$ and $f$ we denote the ({\it{a}}, if $e=0$) tautological
section of minimal self-intersection $s^2=-e$, and a fiber, respectively.
We recall that the classes of $s$ and $f$ generate the Picard group of $\mathbb F_e$, hence,
for any line bundle $\mathcal L$ on $\mathbb F_e$ we can write $\mathcal L =  [\alpha s+ \beta f]$
for some integers $\alpha$, $\beta$.
According to \cite[Corollary 2.18, p. 380]{Ha}, $\mathcal L$ is ample, if and only if it is very ample, if and only if
$\alpha > 0$ and $\beta > \alpha e$. For further properties we refer to \cite[Ch.\ V, Sec.\ 2]{Ha}.
In particular, when $\alpha=2$ (and $\beta > 2e$),
the polarized surface $(\mathbb F_e, \mathcal L)$ is a conic bundle. In this case,
by the projection formula we have
$h^0(\mathcal L) = h^0(\pi_*\mathcal L) =
h^0(S^2 \mathcal E \otimes
\mathcal O_{\mathbb P^1}(\beta))$, where $\pi: \mathbb F_e \to \mathbb P^1$ is the bundle projection
and $S^2$ stands for the second symmetric power.
Note that
$$S^2 \mathcal E \otimes \mathcal O_{\mathbb P^1}(\beta) = \mathcal O_{\mathbb P^1}(\beta)
\oplus \mathcal O_{\mathbb P^1}(\beta-e) \oplus \mathcal O_{\mathbb P^1}(\beta-2e).$$
Moreover, all summands have positive degree, since $\beta > 2e$. Therefore,
\begin{equation} \label{h0}
h^0(\mathcal L) = 3(\beta - e+1).
\end{equation}
It is useful to recall that $\mathbb F_1$ is isomorphic to the projective plane $\mathbb P^2$
blown-up at a point, say $q_0$. If $\tau:\mathbb F_1 \to \mathbb P^2$ stands for this blow-up,
then the tautological section $s$ is just the exceptional curve $\tau^{-1}(q_0)$. It follows that
$\tau^*\mathcal O_{\mathbb P^2}(1) = [s+f]$, and then the general element in the linear system
$|s+ \beta f|$ is the proper transform via $\tau$ of an irreducible plane curve of degree $\beta$ having
a singular point of multiplicity $\beta -1$ at $q_0$.

The following fact will be used often.

\noindent{\it{Remark}} 1. Let $r$ be any positive integer. On $\mathbb F_1$, we have $s \cap \gamma = \emptyset$ for any irreducible curve
$\gamma \in |r(s+f)|$. This simply follows from the fact that $s \cdot (s+f)=0$.

\section{Rational conic fibrations in the setting\\ of polarized surfaces}\label{ample}

Let $X$ be a smooth projective surface and let $\mathcal L$ be any line bundle on $X$.
We say that $(X,\mathcal L)$ is a {\it{pre-conic fibration}} if $X$ is endowed with a fibration
$\pi:X \to B$ over a smooth curve $B$ (i.e.\ $\pi$ is a morphism with connected fibers) such
that the following conditions hold:
\begin{equation}\label{general fiber}
(F,\mathcal L_F) = \big(\mathbb P^1, \mathcal O_{\mathbb P^1}(2)\big) \ \text{for any general fiber}\ F
\ \text{of}\ \pi;
\end{equation}
moreover, singular fibers, if any, are reducible and each consists of
\begin{equation}\label{singular fibers}
\text{two (-1)-curves}\ \ell, \ell'\ \text{with}\ \ell \cdot \ell' = \mathcal L \cdot \ell = \mathcal L \cdot \ell'=1.
\end{equation}
Note that if $\mathcal L$ is ample then \eqref{singular fibers} is a consequence of \eqref{general fiber}.
In this case we say that the polarized surface $(X,\mathcal L)$ is a {\it{conic fibration}}.
We will use the expressions {\it{pre-conic bundle}} and {\it{conic bundle}} to mean that $X$ has no singular fibers,
i.e., that it is a $\mathbb P^1$-bundle. If $B=\mathbb P^1$, then $X$ is a rational surface and
we say that $(X,\mathcal L)$ is a {\it{rational pre-conic fibration}}. For such a pair, let
$g$ be the (virtual) arithmetic genus of
$\mathcal L$. First of all we point out the following fact.
Let $F$ be a fiber of $\pi$. Since $F^2=0$ we have $K_X \cdot F = -2$ by the genus formula. On the other hand,
$\mathcal L \cdot F = 2$. So $(K_X+\mathcal L)\cdot F = 0$, and $(K_X+\mathcal L)\cdot \ell =0$
for any irreducible component $\ell$ of a reducible fiber, by
\eqref{singular fibers}, hence
$K_X + \mathcal L=\pi^* \mathcal O_{\mathbb P^1}(\alpha)$
for some integer $\alpha$. In particular, $(K_X+\mathcal L)^2=0$; furthermore,
$2g-2= \mathcal L \cdot (K_X + \mathcal L) = \mathcal L \cdot \alpha F =2\alpha$, which gives
\begin{equation}\label{for_g}
\mathcal L = -K_X + (g-1)F.
\end{equation}
In particular, we get $\mathcal L^2 = K_X^2 + 4(g-1)$. On the other hand, let $\mu$ be number of singular
fibers of $\pi:X \to \mathbb P^1$. By contracting an irreducible component of every singular fiber of $\pi$
we get a birational morphism from $X$ to an $\mathbb F_e$, factoring through $\mu$ blowing-ups; hence
$K_X^2 = 8 - \mu$, and then
\begin{equation}\label{deg}
\mathcal L^2 = 4(g+1) - \mu.
\end{equation}
This applies in particular, when $\mathcal L$ is ample, to polarized rational conic fibrations, that we simply call
{\it{rational conic fibrations}} when there is no ambiguity. In this case $g=g(X,\mathcal L)$ is the sectional genus and
$d=d(X,\mathcal L)$
is the degree of $(X,\mathcal L)$. Then $d > 0$, and $g \geq 0$.
Since $X$ has Picard number $\rho \geq 2$, we note that $g=0$ if and only if the rational conic fibration $(X,\mathcal L)$
has at the same time the structure of a rational scroll \cite[Corollary 2.3]{LP}, and this can happen only for
$(\mathbb P^1 \times \mathbb P^1, \mathcal O_{\mathbb P^1 \times \mathbb P^1}(1,2))$, $\pi$ being
the projection onto the first factor. We can thus assume that $g \geq 1$.
Moreover, we note that rational conic fibrations with $g=1$ are simply the pairs $(X,-K_X)$,
where $X$ is a del Pezzo surface distinct from $\mathbb P^2$. This explains the interest for case $g=2$.

\begin{prp} \label{mu}
Let $(X,\mathcal L)$ be a rational conic fibration. Then
$\mu \leq 4g+3$. Moreover,
$d \leq 4(g+1)$, with equality if and only if
$$(X,\mathcal L)=\big(\mathbb F_e, [2s+(g+e+1)f]\big)$$
for some integer $0 \leq e \leq g$.
\end{prp}
\noindent
\begin{proof}
The first inequality
is obvious in view of \eqref{deg}, since $d \geq 1$, $\mathcal L$ being ample.
The second inequality
is simply due to the fact that $\mu \geq 0$ and the characterization of the equality
follows from the properties of the Segre--Hirzebruch surfaces.
Once we know that $X = \mathbb F_e$ for some $e \geq 0$, the expression of $\mathcal L$ follows from
the genus formula. Finally, the fact that $s \cdot \mathcal L = -e + g +1 > 0$,
due to the ampleness of $\mathcal L$, implies that $e \leq g$.
\end{proof}

\noindent{\it{Remark}} 2. Clearly, the upper bound for $\mu$ provided by Proposition \ref{mu}
can be improved if the ample line bundle $\mathcal L$ satisfies some additional requirement.
Actually, if $\mathcal L$ is ample and spanned, then necessarily $\mathcal L^2 \geq 2$ and equality
holds if and only if either $(X,\mathcal L) = \big(\mathbb P^1 \times \mathbb P^1, \mathcal O_{\mathbb P^1 \times \mathbb P^1}(1,1) \big)$,
which however is not a conic fibration, or there exists a finite morphism $\varphi:X \to \mathbb P^2$ of degree $2$
such that $\mathcal L = \varphi^*\mathcal O_{\mathbb P^2}(1)$. Let $2b$ be the degree of the branch divisor;
then $K_X = \varphi^*\mathcal O_{\mathbb P^2}(-3+b) = (b-3)\mathcal L$ by the ramification formula and so
the condition $(K_X+\mathcal L)^2=0$ implies $b=2$, i.e. $\varphi$ is the del Pezzo double plane, in which case
$g=1$. As a consequence, if $\mathcal L$ is ample and spanned, then $\mu \leq 6$ if $g=1$, while $\mu \leq 4g+1$
if $g \geq 2$.
We have a stronger improvement assuming that $\mathcal L$ is very ample.
Actually, by restricting the fibration morphism $\pi:X \to \mathbb P^1$ to any smooth curve $C \in |\mathcal L|$ we get
a morphism $\pi|_C:C \to \mathbb P^1$ of degree $2$; so either $g=1$ or $C$ is hyperelliptic. Therefore
\begin{equation}\label{d_g}
\mu \leq 4(g+1) - d_g,
\end{equation}
where $d_g$ denotes the smallest degree of a curve of this type. For instance,
$\mu \leq 5$ ($\mu \leq 7$ resp.) if $g=1$ ($g=2$ resp.). However, though \eqref{d_g} is sharp for $g \leq 2$, this
is no longer true in general. For instance, for $g=3$ we have $d_g = 6$, hence \eqref{d_g} would give $\mu \leq 10$.
Nevertheless, according to \cite[Theorem 4.1,ii)]{Io}, the smallest degree of a rational conic fibration embedded in some projective space is $7$,
hence $\mu \leq 9$.

\smallskip

Consider again a rational conic fibration $(X,\mathcal L)$ where
$\mathcal L$ is simply an ample line bundle. In view of the characterization provided by Proposition \ref{mu}
we can assume that $\pi$ has $\mu \geq 1$ singular fibers. As we said, each of them consists of two $(-1)$-lines meeting at
one point. By contracting one of these two $(-1)$-curves on every singular fiber we thus get a birational morphism
$\eta:X \to X_0$ where $X_0$ is a rational $\mathbb P^1$-bundle, hence $X_0=\mathbb F_e$,
for some $e \geq 0$. Moreover, $\mathcal L = \eta^* \mathcal L_0 - \sum_{i=1}^{\mu} e_i$, where $e_1, \dots ,
e_{\mu}$ are the exceptional curves contracted by $\eta$, and $\mathcal L_0$ is an ample line bundle on $\mathbb F_e$
in view of the Nakai--Moishezon criterion and the ampleness of $\mathcal L$.
Furthermore, the fact that $2 = \mathcal L \cdot F$ for every fiber $F$ of $X$ implies
that also the fibers $f$ of $\mathbb F_e$ have degree $2$ with respect to $\mathcal L_0$. So,
we can write $\mathcal L_0 = [2s+\alpha f]$, for some integer $\alpha$.
Note that $K_X = \eta^*K_{\mathbb F_e} + \sum_{i=1}^{\mu}e_i$. Since $K_{\mathbb F_e}=[-2s-(e+2)f]$, the
sectional genus $g$ of $(X,\mathcal L)$ is given by
\begin{equation} \nonumber
2g-2=\mathcal L \cdot (K_X + \mathcal L) = \mathcal L_0 \cdot (K_{\mathbb F_e}+\mathcal L_0) =[2s+\alpha f] \cdot
(\alpha -2-e)f = 2(\alpha -2-e).
\end{equation}
Thus $\mathcal L_0 = [2s + (e+g+1)f]$ and the ampleness implies $e \leq g$. However, we can say more.
Consider the points $p_i=\eta(e_i)$, $i=1, \dots ,\mu$.
Clearly $p_1, \dots , p_{\mu}$ lie on distinct fibers of $\mathbb F_e$. Suppose that $t$ of these points, e.g. the first $t$, $p_1, \dots , p_t$,
lie on the section $s$ (on the same section $s$ if $e=0$).

\begin{prp} \label{elm}
Let $\eta:X \to \mathbb F_e$ be a birational morphism expressing $X$ as the blow-up of $\mathbb F_e$ at points $p_1, \dots ,p_{\mu}$
lying on distinct fibers, with some of them, say $p_1, \dots , p_t$ with $t \geq 1$ belonging to the $(-e)$-section $s$. Then there exists a birational morphism
$\eta^{\prime}:X \to \mathbb F_{e+1}$ expressing $X$ as the blow-up of $\mathbb F_{e+1}$ at points $p_1^{\prime}, \dots, p_{\mu}^{\prime}$ lying on distinct fibers,
with only $t-1$ of them, say $p_2^{\prime}, \dots p_t^{\prime}$, belonging to the $\big(-(e+1)\big)$-section $s^{\prime}$ of $\mathbb F_{e+1}$, and conversely.
\end{prp}
\begin{proof}
By \cite[Section 2]{N} (see also \cite[Ch.\ V. \S\ 1, pp.\ 85--89]{Sh}) the elementary transformation $\text{\rm{elm}}_{p_1}$ centered at $p_1$, gives rise to a commutative
diagram
\begin{equation}\nonumber
\xymatrix@=30pt{ & X \ar[rd]^{\eta^{\prime}} \ar[ld]_{\eta} & \\
\mathbb F_e & - - \xrightarrow{\text{\rm{elm}}_{p_1}} & \mathbb F_{e+1}. }
\end{equation}
where $\eta^{\prime}:X \to \mathbb F_{e+1}$ is the blowing up of $\mathbb F_{e+1}$ at $p_1^{\prime}, \dots , p_{\mu}^{\prime}$,
with $p_i^{\prime} = \text{\rm{elm}}_{p_1}(p_i)$ for $i=2, \dots , \mu$ and $p_1^{\prime}=\text{\rm{elm}}_{p_1}(f_1)$, where
$f_1$ is the fiber of $\mathbb F_e$ through $p_1$, since $p_1 \in s$. Then $s^{\prime}:=\text{\rm{elm}}_{p_1}(s)$ is the $\big(-(e+1)\big)$-section of $\mathbb F_{e+1}$
and $p_i^{\prime} \in s^{\prime}$ for $i=2, \dots ,t$, while $p_1^{\prime} \not\in s^{\prime}$.
\end{proof}

Up to an iterated application of Proposition \ref{elm}, we can thus assume that $e=1$, i.e., $\eta:X \to \mathbb F_1$,
with $p_1, \dots , p_{\mu}$ lying on distinct fibers. As a consequence, $\mathcal L_0 = [2s + (g+2)f]$.
By the way we note that
$$\mathcal L^2 = \mathcal L_0^2 - \mu = -4e + 4(g+2) - \mu = 4(g+1) - \mu,$$
as already pointed out in \eqref{deg}.

The above discussion can be summed up in the following structure theorem for polarized rational conic fibrations
(compare with \cite[Proposition 3.1, i) for $g=2$ and Theorem 4.1, ii) for $g=3$]{Io} when $\mathcal L$ is very ample).
\begin{thm}\label{structure} Let $(X,\mathcal L)$ be a
rational conic fibration of sectional genus $g$. Then either
\begin{enumerate}
\item $(X,\mathcal L) = (\mathbb F_e, [2s+(g+e+1)f])$ for some $0 \leq e \leq g$, or
\item $X$ has $\mu > 0$ singular fibers, it is obtained by blowing-up $\mathbb F_1$ at points $p_1, \dots, p_{\mu}$
lying on distinct fibers, and $\mathcal L = \eta^*[2s+(g+2)f] -\sum_{i=1}^{\mu} e_i$, where $\eta:X \to \mathbb F_1$
is the blow-up.
\end{enumerate}
\end{thm}
In the following, we will explore some further restrictions on the location of the $p_i$'s
deriving from the ampleness
of $\mathcal L$. First of all we note the following fact.

\begin{prp}\label{t}
Suppose that $p_1, \dots , p_t$ belong to the $(-1)$-section $s$ of $\mathbb F_1$. Then
$t \leq g-1$.
\end{prp}
\begin{proof}
Let $\sigma: =\eta^{-1}(s)$ be the proper transform of $s$ via $\eta$.
Then $\sigma = \eta^* s - \sum_{i=1}^t e_i$, since $s$ is smooth. Therefore its degree with respect to $\mathcal L$ is given by
\begin{equation} \label{degree}
\sigma \cdot \mathcal L = \Big(\eta^* s - \sum_{i=1}^t e_i\Big)\cdot \Big(\eta^* \mathcal L_0 - \sum_{i=1}^{\mu} e_i\Big)=
s \cdot \mathcal L_0 -t = g-t.
\end{equation}
Then the assertion follows from the ampleness of $\mathcal L$.
\end{proof}

An advantage of looking at $\eta: X \to \mathbb F_1$ is that, as we said, $\mathbb F_1$ is $\mathbb P^2$ blown-up at a point,
say $q_0$, the tautological section $s$ of $\mathbb F_1$ being the exceptional curve of
the blow-up $\tau: \mathbb F_1 \to \mathbb P^2$ of $\mathbb P^2$ at $q_0$.
Set $q_i:=\tau(p_i)$ for every $i=1, \dots, \mu$. Composing $\eta$ with $\tau$ we can look at
the birational morphism $\theta:= \tau \circ \eta:X \to \mathbb P^2$, which allows us to regard
$X$ as the projective plane blown-up at the $\mu+1$ points $q_0, q_1, \dots , q_{\mu}$.
Note that $q_0$ cannot be collinear with two other points $q_i$ and $q_j$, otherwise their preimages $p_i$ and $p_j$ via $\tau$
would belong to the same fiber of $\mathbb F_1$, which is not the case.

Moreover, given a point, say $p_1$, on $s$,
the corresponding point $q_1 \in \mathbb P^2$ is an infinitely near point to $q_0$.
According to Proposition \ref{t}, we thus see that at most $g-1$ of the $q_i$'s $(i=1, \dots, \mu)$ can be infinitely near points to $q_0$.

More generally, if $\mu > 0$, we can observe the following fact.

\noindent {\it{Remark}} 3. Up to renaming, suppose that $p_1, \dots, p_{\nu}$
with $\nu \leq \mu$ lie on an irreducible section $s'$ of $\mathbb F_1$. Since $s^{\prime} \in |s+\beta f|$ for some $\beta \geq 0$ (so $s^{\prime}=s$ if and only if $\beta = 0$),
letting $\sigma^{\prime}=\eta^{-1}(s^{\prime})$, the same computation as in \eqref{degree}
gives $\sigma^{\prime} \cdot \mathcal L = 2\beta + g -\nu$, and
therefore, the ampleness of $\mathcal L$ implies that no more than $g+2\beta-1$ of the points $p_i$ can lie on an irreducible section $s' \in |s+\beta f|$ of $\mathbb F_1$.
Looking at the blowing-up $\tau:\mathbb F_1 \to \mathbb P^2$ and recalling what we said in Section \ref{back}, we know that
for $\beta = 1$ ($\beta = 2$ respect.), $s^{\prime}$ is the proper transform of a line in $\mathbb P^2$ not containing $q_0$
(of an irreducible conic passing through $q_0$ respect.). Thus, for $g=1$ the above restriction is equivalent to the requirement that no three
of the $q_i$'s in $\mathbb P^2$ ($i=1, \dots, \mu$)
are collinear and that at most four of them can lie on an irreducible conic passing through $q_0$.
These conditions are well-known in view of the classification of del Pezzo surfaces.
Similarly, for $g=2$, we have that at most three of the $q_i$'s ($i=1, \dots, \mu$) can be collinear and at most five of them can lie on an irreducible conic passing through $q_0$.
In the same way, looking at the proper transform $\gamma$ of a general element in $|2s+2f|$, i.e. the proper transform via $\theta$ of a general conic not passing through $q_0$, we conclude
that no more than $2g+3$ of the $q_i$'s $(i=1,\dots, \mu)$ can lie on a general conic in $\mathbb P^2$ (i.e., no more than $5$ for $g=1$ and no more than $7$ for $g=2$).
Let us note the following consequence. If $g=2$, then $q_0, q_1, \dots , q_6$ cannot lie on an irreducible conic, otherwise
$\mathcal L$ could not be ample.
On the other hand, $q_1, \dots, q_7$ lying on an irreducible conic does not affect the ampleness of $\mathcal L$. However, this prevents $\mathcal L$ from being
very ample (see Remark 4 in Sec.\ \ref{rat_g=2}).

Along the same line we can investigate further restrictions on the $p_i$'s in connection with any irreducible curve $C \in |2s+\beta f|$ for $\beta \geq 3$. In this case, $p_a(C) \geq 1$
and, since $C \cdot f= 2$, if $p_i \in C$, then $\text{mult}_{p_i}(C) \leq 2$, so $C$ can have at worst double points, provided that their number does not exceed $p_a(C)$. For instance, letting
$\beta = 3$ we have $p_a(C)=1$, hence $C$ can have one double point at most. So, if e.\ g. $C$ has a double point at $p_1$ and passes through $p_2, \dots , p_{\nu}$, with
$\nu \leq \mu$, then from
$$C \cdot \mathcal L = \Big(\eta^*(2s+3f)-2e_1 - \sum_{j=2}^{\nu}e_j\Big) \cdot \Big(\eta^*(2s+4f)-\sum_{i=1}^{\mu}e_i\Big) = 9-\nu,$$
we see that if $\mu \geq 9$, the ampleness of $\mathcal L$ prevents that $9$ of the $p_i$'s lie on a curve $C \in |2s+3f|$ having a double point at one of them.
This translates in terms of plane curves by saying that $q_0, q_1, \dots, q_9$ cannot lie on a cubic having a double point at a $q_i$ with $i>0$.

\smallskip
A relevant consequence of \eqref{for_g} is
the ampleness or the very ampleness of $\mathcal L$ for $g \geq 2$, provided that the anticanonical bundle $-K_X$ is good enough.
In fact we have

\begin{prp} \label{veryample}
Let $(X, \mathcal L)$  be a pre-conic fibration for some
line bundle $\mathcal L$, let $F$ be a fiber, and
set $\mathcal A_m:=-K_X + mF$ for any positive integer $m$.
\begin{enumerate}
\item  If $-K_X$ is nef then $\mathcal A_m$ is ample for any $m$.
\item  If $-K_X$ is nef and big then $\mathcal A_m$ is very ample for any $m$.
\end{enumerate}
\end{prp}
Note that $\mathcal A_{g-1}=\mathcal L$, by \eqref{for_g}.

\begin{proof}
We have $\mathcal A_m = \mathcal A_1 + (m-1)F$. So, since the line bundle corresponding to
a fiber is spanned, it is enough to prove both assertions for $\mathcal A := \mathcal A_1$
\cite[Ex.\ 7.5(d), p.\ 169]{Ha}. If $-K_X$ is nef, then $\mathcal A$ is nef, being the sum of two nef line bundles; moreover
$\mathcal A^2 = K_X^2 -2K_X \cdot F + F^2 \geq 4$, since $K_X^2 \geq 0$, $-K_X$ being nef, and
$K_X \cdot F=-2$ by the genus formula.
So, if $\mathcal A$ is not ample then there exists an irreducible curve $C \subset X$ such that
$0 = \mathcal A \cdot C = -K_X \cdot C + F \cdot C$, in view of the Nakai--Moishezon criterion.
Since both summands
are non-negative, this implies that $K_X \cdot C=F \cdot C=0$. The latter
condition says that $C$ is contained in a union of fibers, so, due to the structure of
$(X,\mathcal L)$, it is either a fiber or an irreducible component of a singular fiber,
hence a $(-1)$-curve. In both cases the equality $K_X \cdot C = 0$ cannot hold, and therefore
$\mathcal A$ is ample.

To prove the second assertion, write
$\mathcal A = K_X + M$ where $M = -2K_X+F$.
Note that $M$ is nef, so being both $-K_X$ and $F$. Moreover, $M^2 = 4K_X^2 -4 K_X \cdot F = 4 K_X^2 + 8 \geq 12$, since
$-K_X$ is also big by assumption, and $K_X \cdot F=-2$ as already observed.
We can thus apply Reider's theorem \cite[Theorem]{R}:
if $\mathcal A$ is not very ample,
then $X$ has to contain a divisor $D > 0$ such that one of the following holds:
\begin{enumerate}
\item[i)] $D \cdot M = 0$ and $D^2 = -2$ or $-1$;
\item[ii)] $D \cdot M = 1$ and $D^2 = -1$ or $0$;
\item[iii)]  $D \cdot M = 2$ and $D^2 = 0$.
\end{enumerate}
Let us prove that neither of the cases above may occur.

In case i), due to the expression of $M$, it must be $D \cdot K_X=D \cdot F=0$.
The latter condition says that $D$ is contained in a union of fibres, but this immediately leads to a contradiction:
actually, since every singular fiber consists of two $(-1)$-curves meeting
at one point,
we can write $D = \sum_{j=1}^u s_jF_j + \sum_{i=1}^v r_i\ell_i$ where
the $F_j$'s are fibers, the $\ell_i$'s are irreducible components of singular fibers and $s_1, \dots s_u, r_1, \dots r_v$ are positive integers;
then $D \cdot K_X = -2\sum_{j=1}^u s_j - \sum_{i=1}^v r_i$ cannot be zero.

In case ii), it must be $D \cdot K_X=0$ and $D \cdot F=1$. Moreover, the genus formula implies that $D^2=0$. Since $-K_X$ is big
the Hodge index theorem shows that $D$ has to be numerically trivial, but this contradicts $D \cdot F=1$.

In case iii) we get either: j) $D \cdot K_X=0$ and $D \cdot F=2$, or jj) $D \cdot (-K_X)=1$ and $D \cdot F=0$.
The same argument as in ii) rules out subcase j). On the other hand,
subcase jj) cannot occur in view of the genus formula. This concludes the proof.
\end{proof}

\section{Sectional genus two}\label{rat_g=2}

In this Section $(X,\mathcal L)$ will be a rational conic fibration with $g=2$.
First assume that $\mathcal L$ is simply an ample line bundle.
If $\mu \geq 1$, Proposition \ref{t} says that $t \leq 1$, i.e. at most one of the points $p_1, \dots , p_{\mu}$
can lie on the tautological section $s$.
Recalling the description in terms of the birational morphism $\theta:X \to \mathbb P^2$
provided in Sec.\ \ref{ample}, this can be translated into the following result, which
improves \cite[Remark 3.2]{Io} even in the setting of polarized surfaces.

\begin{prp} \label{uno solo}
If $g=2$ then one at most of the $q_i$'s can be an infinitely near point to $q_0$.
\end{prp}

Furthermore,
by specializing \eqref{for_g}, we have
\begin{equation}\label{fundamental}
\mathcal L = -K_X + F,
\end{equation}
where $F$ stands for a fiber of $X$. Moreover, $d=12-\mu$ by \eqref{deg}.
Relation \eqref{fundamental} allows us to prove
the following result, which provides a partial converse to Proposition \ref{veryample} (2).

\begin{prp} \label{converse}
Let $(X,\mathcal L)$
be a rational conic fibration of sectional genus $2$ with $\mathcal L$ very ample.
Then $-K_X$ is nef and big.
\end{prp}

\begin{proof} If $X = \mathbb F_e$, with $e \leq 2$, then $-K_X$ is obviously ample for $e \leq 1$, and nef and big for $e=2$.
So, according to Theorem \ref{structure} and Remark 2, we can suppose that $X$ is obtained via a blow-up $\eta:X \to \mathbb F_1$
at $1 \leq \mu \leq 7$
points lying on distinct fibers, one at most, say $p_1$, belonging to the minimal section $s$ of $\mathbb F_1$,
and the line bundle $\mathcal L$ is given by $\mathcal L = \eta^*[2s+4f]- \sum_{i=1}^{\mu}e_i$.
Recalling \eqref{h0}, we get $h^0(2s+4f)= 12$.
Thus $h^0(\mathcal L) \geq h^0(2s+4f)- \mu \geq 5$, because $\mu \leq 7$.
Since $\mathcal L \cdot F=2$, we thus get
$\dim(|\mathcal L - F|)= \dim (|\mathcal L|)-3 \geq 1$ and then, due to \eqref{fundamental},
$|-K_X|=|\mathcal L - F|$ is at least a pencil.
Clearly, $K_X^2 = K_{\mathbb P^2}^2-\mu - 1 = 8 - \mu \geq 1$, hence $-K_X$ is big
provided that it is nef.
Suppose, that $-K_X$ is not nef.
Then there exists an irreducible curve $C \subset X$ such that
$C \cdot (\mathcal L - F) = C \cdot (-K_X) < 0$, i.e.,
\begin{equation}\notag
C \cdot \mathcal L < C \cdot F.
\end{equation}
As $C \cdot D < 0$ for every $D \in |-K_X|=|\mathcal L - F|$, necessarily the curve $C$ has to be contained
in the fixed part of $|\mathcal L - F|$, so we can write $|\mathcal L - F| = C + |R|$,
where $R$ is an effective divisor. Thus
\begin{equation}\notag
2 = \mathcal L \cdot F = (\mathcal L - F) \cdot F = C \cdot F + R \cdot F \geq C \cdot F,
\end{equation}
because $F$ is nef. On the other hand, the
ampleness of $\mathcal L$ implies $C \cdot \mathcal L \geq 1$.
Combining this with the above inequalities we get
$$1 \leq C \cdot \mathcal L < C \cdot F \leq 2,$$
and this in turn implies
\begin{equation}\label{exceptional case}
C \cdot \mathcal L = 1 \qquad \text{\rm{and}} \qquad C \cdot F = 2.
\end{equation}
We claim that this exceptional situation cannot happen
$\mathcal L$ being very ample. Suppose \eqref{exceptional case} holds.
Then $C$ is a line on
$X \subset \mathbb P^N$ embedded by $|\mathcal L|$; moreover, $C$ is contained in the linear span
$\langle F \rangle \subset \mathbb P^N$ of every fiber $F$, since it intersects every fiber at two points. As a consequence,
the planes $\langle F \rangle$ constitute a pencil whose axis
is the line $C$ itself, and so $X$ is contained in the $\mathbb P^3$ spanned by that pencil.
But the only smooth surface in $\mathbb P^3$ admitting a structure of conic fibration is
the cubic surface, and this is not compatible with $g=2$. Therefore \eqref{exceptional case} cannot
occur.
\end{proof}

More generally, we have
\begin{prp} \label{more generally}
Let $(X,\mathcal L)$
be a rational conic fibration of sectional genus $2$ and suppose that $\mathcal L$ is ample and spanned.
Then either
\begin{enumerate}
\item[a)] $8 \leq \mu \leq 9$, or
\item[b)] $-K_X$ is nef and big, unless
\begin{equation}\label{special}
\mu=7, t=0,\ \text{and}\ p_1, \dots, p_7\ \text{belong to an irreducible
curve of}\ |2s+2f|.
\end{equation}
\end{enumerate}
\end{prp}

\begin{proof}
By Remark 2 we know that $\mu \leq 9$ since $\mathcal L$ is ample and spanned. Suppose we are not in case a).
Then $\mu \leq 7$ and so we can argue as in the
part of the proof of Proposition \ref{converse}
leading to formula \eqref{exceptional case} under the assumption that $-K_X$ is not nef.
Then, $C$ is a bisecant of the fibration. Moreover, the fact that $\mathcal L_C$ is an ample and spanned line bundle
of degree $1$ on $C$ implies that $C$ is a smooth rational curve.
Consider the birational morphism $\eta:X \to \mathbb F_1$ again and set $\gamma := \eta(C)$;
then $\gamma \subset \mathbb F_1$ is an irreducible curve. We can write
$\gamma \sim as+bf$ for some integers $a$ and $b$, and $C = \eta^*\gamma - \sum_{i=1}^{\mu}\nu_i e_i$,
for suitable non-negative integers $\nu_i=\text{\rm{mult}}_{p_i}(\gamma)$. As $F=\eta^*f$, we have
$$2=C \cdot F = \big(\eta ^*\gamma - \sum_{i=1}^{\mu}\nu_i e_i\big) \cdot \eta^*f = \gamma \cdot f = a,$$
and this implies $0 \leq \nu_i \leq 2$ for every $i=1, \dots, \mu$. So, if $\gamma$ is singular, it
can have at worst double points at some of the $p_i$'s and since $C$ is rational,
their number is exactly $p_a(\gamma)$. By the genus formula we have
$$2p_a(\gamma)-2 = \gamma \cdot (K_{\mathbb F_1}+\gamma) = (2s+bf)\cdot (b-3)f = 2(b-3),$$
hence $p_a(\gamma)=b-2$. Let $\varepsilon$ be the number of the $p_i$'s for which $\nu_i=1$.
Then $b-2 + \varepsilon \leq \mu \leq 7$. Thus,
we get
$$1=C \cdot \mathcal L = \big(\eta^*[2s+bf]-\sum_{i=1}^{\mu}\nu_i e_i\big) \cdot \big(\eta^*[2s+4f]- \sum_{i=1}^{\mu} e_i\big)$$
$$=-4+8+2b-2(b-2)-\varepsilon = 8-\varepsilon.$$
Therefore $7 = \varepsilon = \mu$ and $b-2=0$. So, $\gamma$ is a smooth curve
of $|2s+2f|$ containing all the seven $p_i$'s. In particular, $t=0$ in view of Remark 1, and then
assertion b) in the statement is proved.
\end{proof}

\noindent{\it{Remark}} 4. The special case \eqref{special} arising in Proposition \ref{more generally} is an effective exception to the nefness of $-K_X$. To see this,
referring to the blow-up $\tau:\mathbb F_1 \to \mathbb P^2$ and to $\theta=\tau \circ \eta:X \to \mathbb P^2$,
consider that $[2s+2f] = \tau^* \mathcal O_{\mathbb P^2}(2)$.
Then the exceptional situation in \eqref{special}
corresponds to the fact that the seven points $q_i=\tau(p_i)$ ($i=1, \dots , 7$) of $\mathbb P^2$
lie on an irreducible conic. According to \cite[Th\'eor\`eme 1, p.\ 39]{De}, this fact prevents $-K_X$ from being nef (and hence $\mathcal L$ from being very ample).

\smallskip

Here is an easy application of \eqref{fundamental} and Proposition \ref{converse}. Clearly, in case $t=1$, the proper transform
of the tautological section $s$ via $\eta:X \to \mathbb F_1$ is a $(-2)$-curve which is a line of $(X,\mathcal L)$ and
a section of $X$. In fact we can say more.

\begin{prp} \label{lines}
Let $X \subset \mathbb P^N$ be a rational conic fibration of sectional genus $2$ with $\mu >0$, let
$\eta:X \to \mathbb F_1$ be the morphism expressing $X$ as the blow-up of $\mathbb F_1$ at $p_1, \dots, p_{\mu}$,
and let $C \subset X$ be a line.
Then, either
\begin{enumerate}
\item[1)] $C$ is a component of a singular fiber of $X$, or
\item[2)] $C$ is a $(-2)$-curve and a section of $X$.
\end{enumerate}
\end{prp}

\begin{proof} Let $\mathcal L$ be the hyperplane bundle of $X$ and let $C \subset X$ be a line of $(X,\mathcal L)$.
In view of the relation \eqref{fundamental} we have $1= C \cdot \mathcal L = C\cdot (-K_X) + C \cdot F$.
Clearly, $F$ is nef. Moreover, $-K_X$ is also nef in view of Proposition \ref{converse}, then either
\begin{enumerate}
\item[i)] $C \cdot K_X=-1$ and $C \cdot F=0$, or
\item[ii)] $C\cdot K_X=0$ and $C \cdot F=1$,
\end{enumerate}
and in these two cases $C$ is as in 1) and 2) respectively.
\end{proof}

\section{Characterizing the very ampleness of $\mathcal L$ for $g=2$}

Notation as in Section \ref{ample}. Let $\eta:X \to \mathbb F_1$ be the blowing up of $\mathbb F_1$ at
points $p_1, \dots, p_{\mu}$ with $\mu \leq 11$, lying on distinct fibers, and let $\mathcal L = \eta^* [2s+4f]-\sum_{i=1}^{\mu} e_i$.
Since  $g([2s+4f])=2 > 0$, if $p_1, \dots , p_{\mu}$ are general, then $\mathcal L$ is ample, according
to Yokoyama \cite[Theorem, 1.8]{Y}. Then $(X,\mathcal L)$ is a polarized rational conic fibration of
sectional genus 2. On the other hand, if $p_1, \dots , p_{\mu}$ are not general,
the ``bad curves'' containing some of them
and preventing $\mathcal L$ from being ample include those mentioned in Section \ref{ample}, but
there could be more. In this Section, as a consequence of Proposition \ref{veryample} and Proposition
\ref{converse} we succeed to describe precisely the conditions on the location of
$p_1, \dots , p_{\mu}$ characterizing the very ampleness of $\mathcal L$ for $\mu \leq 7$. To do that we need
to translate the condition that $q_0, q_1, \dots ,q_{\mu}$ are in
almost general position according to \cite[p.\ 39]{De} in terms of the points $p_1, \dots , p_{\mu}$
on $\mathbb F_1$. The relevant fact is that, while the same conditions are obviously sufficient
for the mere ampleness of $\mathcal L$, they turn out also to be necessary except in the geometric situation
specified by \eqref{special}.

So, let $(X,\mathcal L)$ be a rational conic fibration of sectional genus $g=2$ and consider the
blow-up $\theta: X \to \mathbb P^2$ at $q_0, q_1, \dots, q_{\mu}$ where $\mu \leq 7$ (we can assume this
fact since it is a necessary condition for the very ampleness of $\mathcal L$; see Remark 2).
According to \cite[Th\'eor\`eme 1, equivalence of conditions (a) and (d), p.\ 39]{De}, $-K_X$ is nef
if and only if $q_0, q_1, \dots, q_{\mu}$ are in almost general position (see \cite[Definition 1, p.\ 39]{De}).
Note that property $(*)$ in \cite[point b) p.\ 38]{De} is certainly satisfied in the present case
in view of Proposition \ref{uno solo}. This means that the first requirement in \cite[Definition 1, p.\ 39]{De}
is fulfilled, and therefore the points $q_0, q_1, \dots, q_{\mu}$ are in almost general position if and only if
they are
\begin{enumerate}
\item[i)] no four on a line, and
\item[ii)] no seven on a conic.
\end{enumerate}
Before translating these conditions in terms of
the points $p_1, \dots,  p_{\mu}$ on $\mathbb F_1$, we note that if $t=1$ (i.e. $q_1$ is infinitely near to $q_0$), then
a line or a conic through $q_1$ has to contain $q_0$.

So, let us translate the complementary situation of i).
\begin{enumerate}

\item[i-1)] suppose that $q_0,q_{i_1}, q_{i_2}, q_{i_3}$ lie on a line.
This means that the three points $p_{i_1}, p_{i_2}, p_{i_3}$ belong to the same fiber of $\mathbb F_1$ (which is not the case since
$\mathcal L$ is ample);
\item[i-2)] suppose that $q_{i_1}, q_{i_2}, q_{i_3}, q_{i_4}$ are on a line, the $4$-tuple not including $q_0$ (and then not even
$q_1$ if $t=1$). This means that the four points $p_{i_1}, p_{i_2}, p_{i_3}, p_{i_4}$
belong to an irreducible curve in the linear system $|s+f|$ (but this contradicts the ampleness of $\mathcal L$: immediate check).
\end{enumerate}

Therefore no four of the $p_i$'s can lie on an irreducible curve in $|s+f|$. However,
three can lie on such a curve if either $t=0$, or $t=1$ provided that $p_1$ is not one of them (recalling Remark 1).

\smallskip
Next let us translate the complementary situation of ii). Of course this requires $\mu=6$ or $7$.
\begin{enumerate}
\item[ii-1)] suppose that the conic contains $q_0$. This means that six of the points $p_i$ belong to an irreducible curve in the
linear system $|s+2f|$ (an immediate check shows that this contradicts the ampleness of $\mathcal L$);
\item[ii-2)] suppose that the conic misses $q_0$ (and then, necessarily, $t=0$);
this means that the seven points $p_1, \dots, p_7$ belong to an irreducible
curve in the linear system $|2s+2f|$ (this is the special case \eqref{special}) in Proposition \ref{more generally} b), which does not contradict the ampleness of $\mathcal L$).
\end{enumerate}

Here is the main result.

\begin{thm}\label{restrictions}
Let $(X,\mathcal L)$ be a rational conic fibration of sectional genus $2$ with $\mathcal L$ ample and $\mu \leq 7$ singular fibers.
\begin{enumerate}
\item[j)]
Then $(X, \mathcal L)$ is as in point $(2)$ of Theorem $\ref{structure}$, with $g=2$; moreover,
the points $p_1, \dots, p_{\mu}$ satisfy the following conditions: at most one can lie on $s$; no four are on an irreducible curve in the linear system $|s+f|$, and
no six are on an irreducible curve in $|s+2f|$.
Conversely, if these conditions are satisfied for a pre-conic fibration $(X,\mathcal L)$ with $g=2$, then $\mathcal L$ is ample.
\item[jj)] Moreover, if all the above conditions are satisfied, then $\mathcal L$ is spanned.
\item[jjj)] Furthermore, $\mathcal L$ is very ample if and only if the above conditions are satisfied and, in addition, in case $\mu=7$, $t=0$ and
the points $p_1, \dots ,p_{\mu}$ do not belong to an irreducible curve in $|2s+2f|$.
\end{enumerate}
\end{thm}

\begin{proof} j) follows from the above discussion, taking also into account Proposition \ref{t}.
Assertions jj) and jjj) are trivial if $\mu=0$. So let $\mu \geq 1$ and consider the blow-up $\eta: X \to \mathbb F_1$ again. Since $\mathcal L = \eta^*[2s+4f]-\sum_{i=1}^{\mu}e_i$,
we see from \cite[Theorem 4.1]{Be} that the sufficient conditions for the spannedness of $\mathcal L$ are weaker than those listed in j); this implies jj).
Finally, jjj) follows from Proposition \ref{veryample} (2) and Proposition \ref{converse}, taking into account jj) and Proposition \ref{more generally}.
\end{proof}

Note that the conditions in j) are also sufficient for the ampleness of $\mathcal L$ when $\mu \leq 6$. Actually, the ampleness of $\mathcal L$ follows
from Proposition \ref{veryample} (1) and \cite[Th\'eor\`eme 1]{De}, since the conditions in j) express the nefness of $-K_X$ for $\mu \leq 6$. On the other hand, if $\mu = 7$, then either
$-K_X$ is nef and then $\mathcal L$ is ample, or $(X,\mathcal L)$ is as in \eqref{special} (which, however does not prevent $\mathcal L$ from being ample.

\begin{cor} Let $(X,\mathcal L)$ be a polarized rational conic fibration with $g=2$. If $\mu \leq 7$ then $\mathcal L$ is very ample, except
if \eqref{special} holds.
\end{cor}

\noindent
{\it{Remark}} 5. We used Bese's paper \cite{Be} for the spannedness of $\mathcal L$ but we did not for the very ampleness. Actually,
according to \cite[Theorem 4.2]{Be}, $\mathcal L$ would be very ample provided that for any
curve $C \in |xs+yf|$ with $x$ and $y$
nonnegative integers no both zero, such that $x \leq 4$, $y \leq 3$,
\begin{equation} \label{Bese}
\text{\rm {at most}}\ \ \nu :=   x(x-y+3)+y(4-x)-3\ \ \text{\rm{of the}}\  p_i \text{\rm{'s belong to}}\  C.
\end{equation}
Let $\mu=7$. Note that $\nu > 7$ for all admissible pairs $(x,y)$ with $x \geq 3$, hence requirement
\eqref{Bese} is meaningful only for the restricted subset of admissible $(x,y)$ with $x \leq 2$. Moreover,
\eqref{Bese} is obvious when $x=0$ because the seven points lie on distinct fibers; similarly, \eqref{Bese}
is obvious also when $y=0$, since in this case $C=s$ or $2s$ according to whether $x=1$ or $2$, and we know that at most
one of the $p_i$'s can lie on $s$. This reduces the analysis to $11-5=6$ admissible pairs and a close check can be
done. In particular, recalling that $\mu=7$, in addition to the conditions in j) of Theorem
\ref{restrictions}, we see that \eqref{Bese} includes some further restrictions,  like that the seven point cannot lie
on a curve in $|s+3f|$ or in $|2s+3f|$.
Looking at the plane model and using the usual notation, we see that these curves are the proper transforms on $\mathbb F_1$ of a cubic passing through $q_0$ with multiplicity 2 or 1 respectively. However, the fact that these curves do not enter at
all in our conditions (see jjj) of Theorem \ref{restrictions}) is not a trouble once we consider that Bese's conditions are only sufficient conditions (but not necessary) for the very ampleness of $\mathcal L$.
For the remaining $(x,y)$, \eqref{Bese} is satisfied in view of the conditions in j) of Theorem \ref{restrictions}.
The key point, however, is that this happens also if the seven points  lie on an irreducible curve $C \in |2s+2f|$.
Thus $\mathcal L$ would be very ample even in this case according to Bese, but we stress that this is in contrast with
what we proved (this is exactly the exceptional configuration in \eqref{special}).

\section{Rational conic fibrations with $g=2$ containing lines}

Finally, in this Section we characterize the rational conic fibrations $(X,\mathcal L)$ of sectional genus
$2$, with $\mathcal L$ very ample, containing a line $C$ transverse to the fibers.

If $X=\mathbb F_e$ with $e\leq 2$, then $\mathcal L = [2s+(e+3)f]$ by Theorem \ref{structure}. Since $C$ is an irreducible
curve we have either $C=s$ or $C \sim as+bf$ for some integers $a$ and $b$ such that $b \geq ae$. Moreover, $a=1$ or $2$
because $C$ is transverse to the fibers and $C \cdot f \leq 2$ since $C$ is a line and $f$ is a conic. In the former case,
from $1 = C \cdot \mathcal L = s \cdot [2s + (e+3)f] = 3-e$ we get $e=2$. In the latter, we get
$1 = C \cdot \mathcal L = (as+bf) \cdot [2s + (e+3)f] = -2ae+2b+a(e+3) \geq a(e+3) \geq 3$, a contradiction.
Therefore
$e=2$ and $C = s$. On the other hand,
as $s \cdot (2s+5f)=1$ we have that this curve is in fact a line of $(X,\mathcal L)$. Next, suppose that $\mu \geq 1$
and use the description via the birational morphism $\eta$,
recalling that $\mu \leq 7$ since $\mathcal L$ is very ample.
By Proposition \ref{lines} we know that $C$ is a section of $X$ with $C^2=-2$. Then
$\gamma:=\eta(C)$ is a section of $\mathbb F_1$, hence $\gamma \in |s+\alpha f|$ for some integer $\alpha \geq 0$.
Moreover, $C = \eta^*\gamma - \sum_{i=1}^\mu \nu_ie_i$, where
$\nu_i = {\rm{mult}}_{p_i}(\gamma) = C \cdot e_i \leq 1$ for every $i=1, \dots ,\mu$, because both
$C$ and $e_i$ are lines in $X$ embedded in some $\mathbb P^N$ by $|\mathcal L|$.
Let $\varepsilon$ be the number of the $p_i$'s for which $\nu_i=1$. Then
$$-2 = C^2 = \gamma^2 -\varepsilon = 2\alpha -1-\varepsilon.$$
Thus $\alpha = \frac{1}{2}(\varepsilon - 1)$, so that $\varepsilon$ has to be odd. If $\varepsilon=1$
then $\alpha=0$, hence $\gamma=s$ and, up to renaming, $p_1 \in s$. If $\varepsilon > 1$,
then, $\varepsilon = 2\alpha +1 \geq 3$.
On the other hand,
$\varepsilon \leq \mu \leq 7$, hence $\alpha \leq 3$. Therefore
$\varepsilon = 3, 5$ or $7$ according to whether $\alpha = 1, 2$ or $3$, respectively.
Let $\alpha=1$; then $\gamma \in |s+f|$ and it contains three of the $p_i$'s
(clearly, no one of them can lie on $s$, by Remark 1).
Next let $\alpha=2$; then $\gamma \in |s+2f|$ and it contains five of the $p_i$'s. Clearly, here $\mu \geq 5$ and it can happen that $p_1 \in s$, because $s \cdot \gamma = s \cdot (s+2f)=1$.
Finally, if $\alpha=3$, then $\varepsilon = 7 = \mu$, i.e. $\gamma \in |s+3f|$ and it contains all seven points $p_i$'s.
The above discussion proves the following result.

\begin{thm}\label{gamma}
Let $(X,\mathcal L)$ be a rational conic fibration of sectional genus
$2$, with $\mathcal L$ very ample. Then $(X, \mathcal L)$ contains a line $C$ transverse to the fibers
if and only if either
$$(X,\mathcal L, C) = (\mathbb F_2, [2s+5f], s),$$
or $(X,\mathcal L)$ is as in point $(2)$ of Theorem $\ref{structure}$ with $\mu \leq 7$ and $C = \eta^{-1}(\gamma)$, where
$\gamma \subset \mathbb F_1$ is an irreducible section and one of the following conditions is satisfied:
\begin{enumerate}
\item[a)] $\gamma = s$ and it contains $p_1$;
\item[b)] $\gamma \in |s+f|$ and it contains three of the $p_i$'s;
\item[c)] $\gamma \in |s+2f|$ and it contains five of the $p_i$'s;
\item[d)] $\gamma \in |s+3f|$ and it contains $p_1, \dots , p_7$.
\end{enumerate}
\end{thm}

At first glance the statement of Theorem \ref{gamma} does not seem to fit with
\cite[Proposition 8 (2)]{LM}.
But there is a good reason for that since the morphism $\eta:X \to \mathbb F_1$
in point (2) of Theorem \ref{structure} we refer to here is not the same as
the birational morphism $X \to \mathbb F_e$ constructed in \cite{LM}. In fact the above curve $\gamma$ has nothing to do with the curve
denoted by $\ell_0$ in \cite{LM}.

\smallskip
\noindent{\it{Remark}} 6.
i) Consider the four cases occurring in Theorem \ref{gamma} for $\mu \geq 1$. Clearly, in case a) we have $t=1$.
In case b), let $p_{i_1}, p_{i_2}, p_{i_3} \in \gamma$. Referring to the blow-up $\tau:\mathbb F_1 \to \mathbb P^2$
we have that $\gamma = \tau^{-1}(\ell)$, where $\ell \subset \mathbb P^2$ is a line containing three of the $q_i$'s. So three of the $q_i$'s are collinear.
Similarly, in case c), let $p_{i_1}, \dots , p_{i_5} \in \gamma$. Referring to $\tau$ again
we have that $\gamma = \tau^{-1}(G)$, where $G \in \big|\mathcal O_{\mathbb P^2}(2)-q_0 - \sum_{j=1}^{5} q_{i_j}\big|$.
So $q_0$ and five more points $q_i$'s (possibly one of them being infinitely near to $q_0$) are on a conic. As a consequence, in all these cases,
even if $t=0$, the above conditions prevent $(X,\mathcal L)$ from being a Castelnuovo surface, according to the definition in \cite[p.\ 2857]{LM}.
On the other hand, in case d), we have $\gamma = \tau^{-1}(\Gamma)$, where
$\Gamma \in \big|\mathcal O_{\mathbb P^2}(3)-2q_0 - \sum_{i=1}^{7} q_i \big|$.
So
the eight points $q_j$'s lie on a cubic having a double point at $q_0$. In this case
$(X,\mathcal L)$ is in fact a Castelnuovo surface according to \cite{LM}, provided that $t=0$.

ii) Looking over the proof of Theorem \ref{gamma}, suppose that $\gamma \not= s$ and let $\sigma$ be the proper transform of $s$; then $\sigma \not= C$.
Recall that, up to renaming, $p_1$ at most among the $p_i$'s can belong to $s$, by Proposition \ref{uno solo}.
Then $\sigma = \eta^*s - xe_1$, where $x=1$ or $0$ according to whether $p_1 \in s$ or not (i.e., $t=1$ or $0$).
We thus get
$\sigma \cdot \mathcal L = s \cdot (2s+4f) - x = 2-x$. Hence $\sigma$ is either a line or a conic, accordingly.
In particular, $\sigma$ is a conic if and only if $t=0$.
As a consequence, $\sigma \cdot C \leq 2$, since $C$ is a line. Moreover, equality means that the line $C$ lies
in the plane $\langle \sigma \rangle$ spanned by $\sigma$ in the $\mathbb P^N$ where $|\mathcal L|$ embeds $X$. An
easy computation shows that
\begin{equation} \notag
\sigma \cdot C = (\eta^*s - xe_1) \cdot \Big(\eta^*[s+\alpha f]- \sum_{j=1}^{\varepsilon}e_{i_j}\Big) =
\begin{cases}
\alpha -2 & \text{if $p_1 \in s \cap \gamma$},\\
\alpha -1 & \text{otherwise}.
\end{cases}
\end{equation}
This is $2$ if and only if $\alpha =3$, namely in case d) of Theorem \ref{gamma}. Moreover, since
in this case $2= \sigma \cdot C = s \cdot (s+3f) - x = 2-x$, implies $x=0$, we have $t=0$. Therefore
$C \subset \langle \sigma \rangle$ if and only if
$X$ embedded by $|\mathcal L|$ is a Castelnuovo surface of degree $5$ in $\mathbb P^4$.

\smallskip
Now we connect the previous discussions with the inflection points of
$X \subset \mathbb P^N$, linearly normally embedded by $|\mathcal L|$. Let us recall some general facts from
\cite[Section 2]{LM} in the specific case we are dealing with. Let $\mathcal P_X(\mathcal L)$ be the second principal part bundle
of $\mathcal L$ and let
$$j:H^0(X,\mathcal L)\otimes \mathcal O_X \to \mathcal P_X(\mathcal L)$$
be the sheaf homomorphism associating to every section $\sigma \in H^0(X,\mathcal L)$ its second jet $j_x(\sigma)$ evaluated at $x$,
for every $x \in X$.
Let $r:= \max_{x \in X}\{\text{rk}(j_x)\}$ be the maximum rank of $j_x$ on $X$. Clearly, $3 \leq r \leq \min\{N+1, 6\}$ since
$\mathcal P_X(\mathcal L)$ is a vector bundle of rank $6$, $X$ being a surface. A point $x \in X$ is said to an
{\it{inflection point}} if $\text{rk}(j_x) < r$.
This is in accordance with the fact that the {\it{osculating space}} to $X$ at $x$
is defined as $\text{Osc}_x(X):=\mathbb P({\rm{Im}}j_x)$. Then the {\it{inflectional locus}} $\Phi(X)$ of $X$ is defined as
the set of the inflection points.
In view of the definition of the osculating space, the osculating hyperplanes to $X$ at $x$ can be
regarded as the elements of the linear system $|\mathcal L- 3x|$ (hyperplane sections of $X$ having a singular
point of multiplicity $\geq 3$ at $x$). As a consequence,
$\dim(|\mathcal L - 3x|) + \dim \big(\text{Osc}_x(X)\big) = N-1$, or equivalently
$\text{rk}(j_x)= N - \dim(|\mathcal L - 3x|)$. In particular, we see that
\begin{equation}\label{why}
\text{\rm{if}}\ |\mathcal L - 3x| = \emptyset,\ {\rm{then}}\ x \not\in \Phi(X).
\end{equation}
Coming back to the situation in Theorem \ref{gamma}, note that
$$N = h^0(\mathcal L)-1 \geq h^0(2s+4f) -1 - \mu = 11-\mu$$
due to \eqref{h0}. Moreover, $N \geq 4$ since $\mathcal L$ is very ample.
If $N \geq 5$, then the line $C$
is always contained in $\Phi(X)$, according to \cite[Theorem 5 (2)]{LM}.
If $N=4$, which implies $\mu=7$, we have to analyze all possible cases a)--d) in Theorem \ref{gamma} to decide whether the
line $C$ is contained or not in $\Phi(X)$. We can argue as follows.
Fix $x \in C$ and let $F_x$ be the fiber of $X$ passing through $x$. Since $C$ and $F_x$ are a line and a conic
containing $x$, they are fixed components
of $|\mathcal L - 3x|$ \cite[Lemma 2(1)]{LM}. Thus,
$$\dim(|\mathcal L - 3x|) = \dim(|\mathcal L - F_x - C -x|) = \dim(|-K_X-C-x|),$$
recalling \eqref{fundamental}.
Keeping the notation used in the proof of Theorem
\ref{gamma}, we can write
$C \sim \eta^*(s+\alpha f) - \sum_{j=1}^{\varepsilon} e_{i_j}$,
where $\varepsilon = 2\alpha + 1$, with $\alpha = 0,1,2,3$ according to cases
a), b), c) and d) respectively in Theorem \ref{gamma}.
Letting $P:=\{p_1, \dots ,p_{\mu}\}$, we thus see that $P \setminus \gamma$ consists
of $\mu - \varepsilon = \mu-1 -2\alpha$ points.
Thus
$$\dim(|-K_X -C|) = \dim\Big(\big|\eta^*\big(s+(3-\alpha)f\big) - \sum_{k=1}^{\mu-\varepsilon} e_{i_k}\big|\Big)$$
$$\qquad \qquad \qquad \quad \geq  \dim(|s+(3-\alpha)f|) - (\mu-1 - 2\alpha),$$
equality holding if and only if
\begin{equation}\label{independent}
P \setminus \gamma\ \text{\rm{imposes linearly independent linear conditions on}}\ |s+ (3-\alpha)f|.
\end{equation}
Note that this certainly happens in case d) of Theorem \ref{gamma}, because $P \setminus \gamma$ is empty.
Now, a straightforward computation shows that $h^0\big(s+(3-\alpha)f\big) = 7-2\alpha$, hence
$\dim(|-K_X -C|)= 0$, provided that \eqref{independent} holds and $\mu=7$. This in turn implies that
\begin{equation}\label{empty}
|-K_X-C-x| = \emptyset
\end{equation}
unless $x \in D \cap C$, where $D$ is the only divisor in $|-K_X - C|$.
But $D$ is a section and $D \cdot C=2$, hence \eqref{empty} holds for the general point $x \in C$.
Then, as a consequence of \eqref{why}, $C$ is not contained in $\Phi(X)$. So we have the following result.
\begin{prp} \label{final}
Let $(X, \mathcal L)$ and $C$ be as in Theorem $\ref{gamma}$ and suppose that $N=4$ (hence $\mu=7$). If \eqref{independent} holds,
then the line $C$ is not contained in the inflectional locus of $X$. In particular this applies to case \rm{d)}.
\end{prp}

{\it{Remark}} 7. In fact, as to case d) we proved in \cite[Theorem 18]{LM} that when $X \subset \mathbb P^4$
is a Castelnuovo surface, then the inflectional locus is finite and consists of $75$ points, in general. We use this opportunity to point out that
unlike we said in \cite[Example b), p.\ 2870]{LM} Castelnuovo surfaces, as defined in \cite[p.\ 2857]{LM} do not exhaust the class of surfaces
of sectional genus $2$ in $\mathbb P^4$. Actually, there exist surfaces with $t=0$ that are not Castelnuovo surfaces as well as examples of surfaces with $t=1$.

\medskip

The following example offers an interesting situation in connection with Proposition \ref{lines}, Theorem \ref{gamma} and Proposition \ref{final}.

{\it{Example}}. Let $X$ be the surface obtained by blowing-up $\mathbb F_1$ at $7$ points $p_1, \dots , p_7$ lying
on an irreducible curve $\gamma \in |s+3f|$ and sufficiently general, with $p_1 \in s$, and let $\mathcal L = \eta^*[2s+4f]-\sum_{i=1}^7 e_i$, with obvious
meaning of the symbols. Then $\mathcal L$ is very ample, $|\mathcal L|$ embeds $X$ in $\mathbb P^4$ as a quintic surface
and $(X,\mathcal L)$ is a rational conic fibration with $g=2$ and $t=1$. In particular,
it is not a Castelnuovo surface in the sense of \cite[p.\ 2857]{LM}. Let $\widetilde{s}$ and $C$ be the proper transforms on $X$
of $s$ and of $\gamma$, respectively. Both curves are lines transverse to the fibers of $X$, and in fact they are $(-2)$-curves. They correspond
to case a) and to case d) respectively in Theorem \ref{gamma}.
Since $\eta^*s = \widetilde{s} + e_1$ and $\eta^* \gamma = C + \sum_{i=1}^7 e_i$, one immediately sees that $\widetilde{s}$ and $C$ intersect at one point, say $y_1$. Similarly,
the line $e_1$, which is a component of a singular fiber, intersects both $C$ and $\widetilde{s}$. Set $y_2 := e_1 \cap C$ and $y_3 := e_1 \cap \widetilde{s}$.
Following the argument used to prove Proposition \ref{final}, we see that $-K_X-C = \eta^*s = \widetilde{s}+e_1$. Similarly,
$-K_X-\widetilde{s} = \eta^*[s+3f]-\sum_{i=2}^7 e_i = C + e_1$. According to Proposition \ref{final} the conclusion is that no $x \in C \cup \widetilde{s}$ can be an inflection point
of $X$ except the three points $y_j$, $j=1,2,3$, which indeed are such.

\smallskip

{\bf Acknowledgements}: The authors are grateful to the referee for useful remarks. The first author is a member of G.N.S.A.G.A. of the Italian INdAM. He would like to thank the PRIN 2015 Geometry of Algebraic Varieties and the University of Milano for partial support.
The second author wants to thank the Spanish Ministry of Science, Innovation and Universities (Project MTM 2015-65968-P ``Geometr\'{i}a algebraica y anal\'{i}tica y aplicaciones'').

\end{document}